\documentclass[12pt]{article}
\usepackage{a4wide}

\usepackage[mathscr]{euscript}

\DeclareFontFamily{OT1}{ptm}{}
\DeclareFontShape{OT1}{ptm}{m}{n} { <-> ptmr}{}
\DeclareFontShape{OT1}{ptm}{m}{it}{ <-> ptmri}{}
\DeclareFontShape{OT1}{ptm}{m}{sl}{ <->ptmro}{}
\DeclareFontShape{OT1}{ptm}{m}{sc}{ <-> ptmrc}{}
\DeclareFontShape{OT1}{ptm}{b}{n} { <-> ptmb}{}
\DeclareFontShape{OT1}{ptm}{b}{it}{ <-> ptmbi}{}     
\DeclareFontShape{OT1}{ptm}{bx}{n} {<->ssub * ptm/b/n}{}
\DeclareFontShape{OT1}{ptm}{bx}{it}{<->ssub * ptm/b/it}{}

\DeclareSymbolFont{bold}{OT1}{ptm}{b}{n}
\DeclareMathAlphabet{\mathbf}{OT1}{ptm}{b}{n}  
\DeclareMathAlphabet{\mathrm}{OT1}{ptm}{m}{n}

\DeclareFontFamily{OT1}{psy}{}      
\DeclareFontShape{OT1}{psy}{m}{n}{ <-> s * [0.9] psyr}{}
\DeclareFontFamily{OMS}{ptm}{}     
\DeclareFontShape{OMS}{ptm}{m}{n}{ <8> <9> <10> gen * cmsy }{}
\DeclareFontFamily{OMS}{cmtt}{}     
\DeclareFontShape{OMS}{cmtt}{m}{n}{ <8> <9> <10> gen * cmsy }{}

\SetSymbolFont{operators}{normal}{OT1}{ptm}{m}{n}   
\SetSymbolFont{operators}{bold}{OT1}{ptm}{b}{n}     
\DeclareSymbolFont{emsy}{OT1}{ptm}{m}{it}
\DeclareSymbolFont{emsr}{OT1}{ptm}{m}{n}
\DeclareSymbolFont{emcmr}{OT1}{cmr}{m}{n}   
\DeclareSymbolFont{emsymb}{OT1}{psy}{m}{n}  
\DeclareMathSymbol a{\mathalpha}{emsy}{"61}
\DeclareMathSymbol b{\mathalpha}{emsy}{"62}
\DeclareMathSymbol c{\mathalpha}{emsy}{"63}
\DeclareMathSymbol d{\mathalpha}{emsy}{"64}
\DeclareMathSymbol e{\mathalpha}{emsy}{"65}
\DeclareMathSymbol f{\mathalpha}{emsy}{"66}
\DeclareMathSymbol g{\mathalpha}{emsy}{"67}
\DeclareMathSymbol h{\mathalpha}{emsy}{"68}
\DeclareMathSymbol i{\mathalpha}{emsy}{"69}
\DeclareMathSymbol j{\mathalpha}{emsy}{"6A}
\DeclareMathSymbol k{\mathalpha}{emsy}{"6B}
\DeclareMathSymbol l{\mathalpha}{emsy}{"6C}
\DeclareMathSymbol m{\mathalpha}{emsy}{"6D}
\DeclareMathSymbol n{\mathalpha}{emsy}{"6E}
\DeclareMathSymbol o{\mathalpha}{emsy}{"6F}
\DeclareMathSymbol p{\mathalpha}{emsy}{"70}
\DeclareMathSymbol q{\mathalpha}{emsy}{"71}
\DeclareMathSymbol r{\mathalpha}{emsy}{"72}
\DeclareMathSymbol s{\mathalpha}{emsy}{"73}
\DeclareMathSymbol t{\mathalpha}{emsy}{"74}
\DeclareMathSymbol u{\mathalpha}{emsy}{"75}
\DeclareMathSymbol v{\mathalpha}{emsy}{"76}
\DeclareMathSymbol w{\mathalpha}{emsy}{"77}
\DeclareMathSymbol x{\mathalpha}{emsy}{"78}
\DeclareMathSymbol y{\mathalpha}{emsy}{"79}
\DeclareMathSymbol z{\mathalpha}{emsy}{"7A}
\DeclareMathSymbol A{\mathalpha}{emsy}{"41}
\DeclareMathSymbol B{\mathalpha}{emsy}{"42}
\DeclareMathSymbol C{\mathalpha}{emsy}{"43}
\DeclareMathSymbol D{\mathalpha}{emsy}{"44}
\DeclareMathSymbol E{\mathalpha}{emsy}{"45}
\DeclareMathSymbol F{\mathalpha}{emsy}{"46}
\DeclareMathSymbol G{\mathalpha}{emsy}{"47}
\DeclareMathSymbol H{\mathalpha}{emsy}{"48}
\DeclareMathSymbol I{\mathalpha}{emsy}{"49}
\DeclareMathSymbol J{\mathalpha}{emsy}{"4A}
\DeclareMathSymbol K{\mathalpha}{emsy}{"4B}
\DeclareMathSymbol L{\mathalpha}{emsy}{"4C}
\DeclareMathSymbol M{\mathalpha}{emsy}{"4D}
\DeclareMathSymbol N{\mathalpha}{emsy}{"4E}
\DeclareMathSymbol O{\mathalpha}{emsy}{"4F}
\DeclareMathSymbol P{\mathalpha}{emsy}{"50}
\DeclareMathSymbol Q{\mathalpha}{emsy}{"51}
\DeclareMathSymbol R{\mathalpha}{emsy}{"52}
\DeclareMathSymbol S{\mathalpha}{emsy}{"53}
\DeclareMathSymbol T{\mathalpha}{emsy}{"54}
\DeclareMathSymbol U{\mathalpha}{emsy}{"55}
\DeclareMathSymbol V{\mathalpha}{emsy}{"56}
\DeclareMathSymbol W{\mathalpha}{emsy}{"57}
\DeclareMathSymbol X{\mathalpha}{emsy}{"58}
\DeclareMathSymbol Y{\mathalpha}{emsy}{"59}
\DeclareMathSymbol Z{\mathalpha}{emsy}{"5A}
\DeclareMathSymbol{\bullet}{\mathalpha}{emsymb}{"B7}
\DeclareMathSymbol{\regis}{\mathalpha}{emsymb}{"D2}
\def\Bullet{\leavevmode\unkern{$\m@th\bullet$}\kern.32em\ignorespaces}
\def\Regis{\leavevmode\raise.5ex\hbox{$\m@th\regis$}}
\DeclareMathSymbol +{\mathbin}{emcmr}{`+}
\DeclareMathSymbol ={\mathrel}{emcmr}{`=}  
\DeclareMathSymbol{\Gamma}{\mathalpha}{emcmr}{"00}
\DeclareMathSymbol{\Delta}{\mathalpha}{emcmr}{"01}
\DeclareMathSymbol{\Theta}{\mathalpha}{emcmr}{"02}
\DeclareMathSymbol{\Lambda}{\mathalpha}{emcmr}{"03}
\DeclareMathSymbol{\Xi}{\mathalpha}{emcmr}{"04}
\DeclareMathSymbol{\Pi}{\mathalpha}{emcmr}{"05}
\DeclareMathSymbol{\Sigma}{\mathalpha}{emcmr}{"06}
\DeclareMathSymbol{\Upsilon}{\mathalpha}{emcmr}{"07}
\DeclareMathSymbol{\Phi}{\mathalpha}{emcmr}{"08}
\DeclareMathSymbol{\Psi}{\mathalpha}{emcmr}{"09}
\DeclareMathSymbol{\Omega}{\mathalpha}{emcmr}{"0A}
\DeclareMathSizes{7.6}{8}{6}{5}

\usepackage{amsfonts,amsthm}
\usepackage{amssymb}
\usepackage[leqno]{amsmath}

\usepackage{tikz-cd}

\usepackage{mathabx}

\usepackage{sseq}
\usepackage{rotating}
\usepackage{bbm}
\usepackage[all,cmtip]{xy}
\usepackage{amscd}
\usepackage{caption}
\usepackage{subcaption}

\usepackage{tikz}
\usepackage{enumitem}

\usepackage{mathrsfs}
\usepackage{color}
\usepackage{xcolor}

\usepackage[latin1]{inputenc}
\usepackage{graphicx}
\usepackage{dsfont}
\usepackage[colorlinks]{hyperref}
\hypersetup{%
  citecolor=blue,%
}

\newtheorem{theorem}{Theorem}[section]

\newtheorem{lemma}[theorem]{Lemma}
\newtheorem{proposition}[theorem]{Proposition}
\newtheorem{corollary}[theorem]{Corollary}

\theoremstyle{remark}
\newtheorem{remark}[theorem]{Remark}

\theoremstyle{definition}

\newtheorem{example}[theorem]{Example}

\theoremstyle{theorem}
\newtheorem{TheoremA}{Theorem}

\newtheorem*{maintheorem}{Main Theorem}
\newtheorem*{problem1}{Problem}

\newcommand*{\longhookrightarrow}{\ensuremath{\lhook\joinrel\relbar\joinrel\rightarrow}}

\newcommand{\ow}{\omega}
\newcommand{\p}{\partial}

\newcommand{\C}{{\mathbb{C}}}

\newcommand{\R}{{\mathbb{R}}}
\newcommand{\Q}{{\mathbb{Q}}}
\newcommand{\Z}{{\mathbb{Z}}}

\renewcommand{\epsilon}{\varepsilon}
\DeclareMathOperator{\CZ}{CZ}
\DeclareMathOperator{\rel}{rel}
\DeclareMathOperator{\std}{std}

\DeclareMathOperator{\cyl}{cyl}

\DeclareMathOperator{\reg}{reg}

\DeclareMathOperator{\tr}{trace}

\DeclareMathOperator{\pt}{pt}
\DeclareMathOperator{\Cl}{Clif}

\DeclareMathOperator{\id}{id}
\DeclareMathOperator{\im}{Im}
\DeclareMathOperator{\ind}{ind}

\DeclareMathOperator{\Fix}{Fix}
\DeclareMathOperator{\Aut}{Aut}

\DeclareMathOperator{\can}{can}

\usepackage{graphicx}
\usepackage{authblk}

\begin{document}
\title{\bf  \Large
Unknottedness of real Lagrangian tori in $S^2\times S^2$
}

\makeatletter
\newcommand{\subjclass}[2][2010]{%
  \let\@oldtitle\@title%
  \gdef\@title{\@oldtitle\footnotetext{#1 \emph{Mathematics Subject Classification.} #2.}}%
}
\newcommand{\keywords}[1]{%
  \let\@@oldtitle\@title%
  \gdef\@title{\@@oldtitle\footnotetext{\emph{Key words and phrases.} #1.}}%
}
\makeatother

\author
{
Joontae Kim
\thanks
{
School of Mathematics, Korea Institute for Advanced Study, 85 Hoegiro, Dongdaemun-gu, Seoul 02455, Republic of Korea,  \texttt{joontae@kias.re.kr}}
%\thanks{
%Institut de Math\'ematiques, Universit\'e de Neuch\^atel, Rue \'Emile Argand 11, 2000 Neuch\^atel, Switzerland, \texttt{joontae.kim@unine.ch}}
}
%\subjclass{53D12, 53D35, 54H25}
%\affil{\normalsize
%School of Mathematics, Korea Institute for Advanced Study, 85 Hoegiro, Dongdaemun-gu, 02455 Seoul, Republic of Korea\\ \texttt{joontae@kias.re.kr}}
\date{Recent modification; \today}
\date{}

\setcounter{tocdepth}{2}
\numberwithin{equation}{section}
\maketitle

\begin{abstract}
We prove the Hamiltonian unknottedness of real Lagrangian tori in the monotone $S^2\times S^2$, namely any real Lagrangian torus in $S^2\times S^2$ is Hamiltonian isotopic to the Clifford torus. The proof is based on a neck-stretching argument, Gromov's foliation theorem, and the Cieliebak--Schwingenheuer criterion.
\end{abstract}

%\tableofcontents

\section{Introduction}
%A \emph{symplectic} manifold $(M,\ow)$ is an even dimensional smooth manifold $M$ equipped with a closed non-degenerate 2-form $\ow$. 
An even dimensional smooth manifold $M$ equipped with a closed non-degenerate 2-form $\ow$ is a \emph{symplectic manifold}. By Darboux's theorem \cite[Theorem 3.2.2]{MS}, symplectic manifolds are locally standard, and hence only global properties in symplectic topology are interesting; in particular, the study of middle dimensional submanifolds along which the symplectic form vanishes, namely \emph{Lagrangian submanifolds}.

In 1986, as one of the first steps in symplectic topology \cite[Section 6]{ArnFirst}, Arnold proposed the \emph{Lagrangian knot problem} asking whether two given Lagrangians are isotopic.
As formulated in a survey of Eliashberg--Polterovich \cite{EPfour}, there are different flavors of isotopy, namely \emph{smooth}, \emph{Lagrangian} and \emph{Hamiltonian}.
Hamiltonian isotopies are Lagrangian, and Lagrangian isotopies are smooth.
Two Lagrangians are said to be \emph{unknotted} if they are isotopic to each other in one of these three ways.

A remarkable result of Gromov \cite{Gromov} says that there are no closed \emph{exact} Lagrangians in $(\R^{2n},\sum_{i=1}^ndx_i\wedge dy_i)$, and hence the extensive study of Lagrangian tori in $\R^{2n}$ has been made for a long time.
Chekanov \cite{Chekanov} first constructed a \emph{monotone} Lagrangian torus in $\R^{2n}$ for $n\ge 2$, which is Lagrangian isotopic while not Hamiltonian isotopic to the \emph{Clifford torus} $\mathbb{T}^n_{\Cl}=\times_n S^1$, i.e., products of circles in $\R^2$ of equal radius.
This exhibits that a Lagrangian isotopy cannot be in general deformed into a Hamiltonian isotopy.
His result is even more interesting since the classical invariants cannot detect this phenomenon.
Indeed, the Audin conjecture \cite[Section 6.4]{Audin}, which is proved by Polterovich \cite{Polaudin} and Viterbo \cite{Vit} in dimension 4 and Cieliebak--Mohnke \cite{CM} in any dimension, says that the minimal Maslov number (one of the classical invariants) of any Lagrangian torus in $\R^{2n}$ is always \emph{two}, so exotic monotone Lagrangian tori are hard to discover.
Auroux \cite{Auroux} constructed infinitely many monotone Lagrangian tori in $\R^6$ up to Hamiltonian isotopy, while all of them are Lagrangian isotopic. We refer to the work of Dimitroglou Rizell--Evans \cite[Corollary C]{DE} about the smooth unknottedness of monotone Lagrangian tori inside $\R^{2n}$ for $n\ge 5$ odd.

Since Lagrangian 2-planes in $\R^4$ that are asymptotically linear are trivial by Eliashberg--Polterovich \cite{EPlocal} (see also \cite{EPunknot}), one may expect reasonable unknottedness results in symplectic 4-manifolds. 
Symplectic Field Theory developed by Eliashberg--Givental--Hofer \cite{EGH} has become a core technique to address the \emph{Lagrangian unknottedness problems}. 
By a neck-stretching argument \cite{BEHWZ}, Hind \cite{HindS2S2} showed the Hamiltonian unknottedness of spheres in $S^2\times S^2$.
%One should be aware that a Lagrangian isotopy of \emph{spheres} is extended to a Hamiltonian isotopy as $H_{\dR}^1(S^2)=0$, see \cite[Exercise 6.1.A]{Polbook}. 

In contrast to spheres, the Hamiltonian unknottedness of monotone Lagrangian \emph{tori} in $S^2\times S^2$ fails. Chekanov--Schlenk \cite{ChekanovSchlenk} found a monotone Lagrangian torus in $S^2\times S^2$ which is not Hamiltonian isotopic to the \emph{Clifford torus} $\mathbb{T}_{\Cl}=S^1\times S^1$, defined as the product of the equators. This torus, 
in another description, was also found by Entov--Polterovich \cite[Example 1.22]{EPtorus}. Vianna \cite{Vian} showed that there are infinitely many Hamiltonian isotopy classes of monotone Lagrangian tori in $S^2\times S^2$. Dimitroglou Rizell--Goodman--Ivrii \cite[Theorem A]{DGI} established the \emph{Lagrangian unknottedness} of tori in $S^2\times S^2$, namely Lagrangian tori in $S^2\times S^2$ are unique up to Lagrangian isotopy. As a result, infinitely many monotone Lagrangian tori in $S^2\times S^2$ are \emph{Hamiltonianly knotted}.

In this paper, we are interested in a class more rigid than monotone Lagrangian tori, namely real Lagrangian tori.
By a \emph{real Lagrangian submanifold} $L$ in a symplectic manifold $(M,\ow)$ we mean a Lagrangian submanifold that is the fixed point set of an \emph{antisymplectic involution} $R$ of $M$, i.e., $R^2=\id_M$ and $R^*\ow=-\ow$.
The fixed point set $\Fix(R)=\{x\in M \mid R(x)=x\}$ of an antisymplectic involution $R$ of a symplectic manifold $(M,\ow)$ is Lagrangian if it is nonempty.
If $(M,\ow)$ is monotone, then every real Lagrangian of $M$ is monotone, see Lemma~\ref{lem: realmono}.
Recall that the Clifford torus $\mathbb{T}_{\Cl}=S^1\times S^1$ is a real Lagrangian torus in $S^2\times S^2$ whose antisymplectic involution $R_{\Cl}$ is given by the product of the reflection of $S^2$ fixing the equator.

The main result of this paper is to prove the Hamiltonian unknottedness of real Lagrangian tori in $S^2\times S^2$ in contrast to the monotone case.
\begin{maintheorem}
	Any real Lagrangian torus in $S^2\times S^2$ is Hamiltonian isotopic to the Clifford torus $\mathbb{T}_{\Cl}$.	
\end{maintheorem}
As we discussed, being real plays a key role, and the result shows a non-trivial phenomenon of Hamiltonian unknottedness.
An immediate consequence of the main theorem together with a known result \cite[Proposition B]{KimUniq} is the complete classification of real Lagrangian submanifolds in $S^2\times S^2$.
\begin{TheoremA}\label{thm: thmA}
Any real Lagrangian submanifold in $S^2\times S^2$ is Hamiltonian isotopic to either the antidiagonal sphere $\overline{\Delta}$ or the Clifford torus $\mathbb{T}_{\Cl}$.
\end{TheoremA}
From a real symplectic perspective, the study of real Lagrangian tori in $S^2\times S^2$ is the simplest non-trivial case in dimension 4. For topological reasons, there is no closed real Lagrangian in $\R^{2n}$ at all. Known monotone symplectic 4-manifolds containing real Lagrangian tori are $S^2\times S^2$ and the three-fold monotone blow-up of $\C P^2$, see \cite[Theorem D]{BKM}.
In \cite{KimChek}, it was proved that the Chekanov--Schlenk torus in $S^2\times S^2$ is not real, from which we believed that our main result holds. We refer to a recent work of Brendel \cite{Brendel} that extends the result in \cite{KimChek}.

In order to prove the main theorem, we employ Gromov's foliation theorem (Section~\ref{sec: Gromov}) together with the Cieliebak--Schwingenheuer criterion (Section \ref{sec: CS}).
Gromov's strategy was to foliate the monotone symplectic manifold $S^2\times S^2$ by a family of $J$-holomorphic spheres of minimal symplectic area, and hence to reduce symplectic questions to two dimensional fibered versions.
In particular, when one deforms the split complex structure $i\oplus i$ on $S^2\times S^2$ to \emph{any} tame almost complex structure $J$, two transversal foliations by $J$-holomorphic spheres still exist (even smoothly depending on $J$), and yield a symplectic $S^2$-fibration of $S^2\times S^2$ over $S^2$.
In the case of studying a monotone Lagrangian torus, the nicest situation is when a fiber symplectic sphere intersects the torus along a circle or does not intersect at all.
Cieliebak--Schwingenheuer provided a criterion for the Hamiltonian unknottedness of tori in $S^2\times S^2$ by means of this fibered structure.
In particular, the existence of two suitable symplectic sections of the fibration is the precise condition for a given monotone Lagrangian torus being Hamiltonianly deformable into the Clifford torus $\mathbb{T}_{\Cl}$.
In general, one symplectic section always exists (essentially by an SFT argument, see Proposition \ref{prop: displace}), but the second symplectic section may be missing.
For real Lagrangian tori, an antisymplectic involution will provide the second symplectic section as desired. To realize this heuristic argument precisely, we should deal with an almost complex structure that is compatible with an antisymplectic involution.

	It might be interesting to mention interactions between real symplectic topology and real algebraic varieties. We refer to \cite{realenriq, KharlamovD, KharlamovMilan} for expositions about the topology of real algebraic varieties.
	Inspired by the notion of \emph{quasi-simplicity} for real algebraic varieties~\cite{KharlamovMilan}, we formulate a counterpart in symplectic topology. A closed monotone symplectic manifold $(M,\ow)$ is called \emph{symplectically quasi-simple} if the diffeomorphism type of \emph{connected} real Lagrangian submanifolds of $M$ uniquely determines its Hamiltonian isotopy class.
	\begin{problem1}
		Show that symplectic del Pezzo surfaces are symplectically quasi-simple. 
	\end{problem1}
Notice that one has to impose a condition on the homology classes of real Lagrangians if necessary. It is known that the problem is true for $\C P^2$ \cite[Proposition A]{KimUniq} and $S^2\times S^2$ (Theorem~\ref{thm: thmA}). Note that quasi-simplicity in real algebraic varieties is known for all real del Pezzo surfaces by Degtyarev--Itenberg--Kharlamov \cite{realenriq, Kharquasi}. Together with Brendel and Moon \cite[Theorem D]{BKM}, we obtain the complete list of diffeomorphism types of \emph{connected} real Lagrangians in \emph{toric} symplectic del Pezzo surfaces. For symplectic results, we refer to the works of Evans \cite{Evans} and Seidel \cite{Seidelknotting}, which deal with the (un)knotting problems of Lagrangian spheres in del Pezzo surfaces. See also the work of Borman--Li--Wu \cite{BLW}. Finally, we mention a work of Welschinger~\cite{Welsch} which might be related to the problem. In that paper, he defines an invariant under deformation of real symplectic 4-manifolds, called \emph{Welschinger invariant}.

This paper is organized as follows. In Section \ref{sec: sec2}, we give results relevant for the proof of the main theorem. In Section \ref{sec: sec3}, we prove the main theorem.

%\paragraph{Organization of the paper.} In Section \ref{sec: sec2}, we give main tools, including Gromov's theorem and Cieliebak--Schwingenheuer's criterion.  In Section \ref{sec: sec3}, we prove the main theorem.

\section{Structural results for real Lagrangian tori in $S^2\times S^2$}\label{sec: sec2}
Recall that a symplectic manifold $(M,\ow)$ is \emph{monotone} if there exists $C>0$ such that $c_1(M)=C\cdot [\ow]$ and that a Lagrangian submanifold $L$ in $M$ is \emph{monotone} if there exists $C'>0$ such that $\mu_L(\beta)=C'\cdot \ow(\beta)$ for all $\beta\in \pi_2(M,L)$, where $\mu_L\colon \pi_2(M,L)\to \Z$ denotes the Maslov class of $L$, see \cite[Definition 3.4.4]{MS}.
\begin{lemma}\label{lem: realmono}
Every real Lagrangian in a monotone symplectic manifold $(M,\ow)$ is monotone.
\end{lemma}
\begin{proof}
Let $L=\Fix(R)$ be a real Lagrangian in $M$ for an antisymplectic involution $R$. Note that $S^2$ is the union of two closed discs $D_1$ and $D_2$ such that $D_2=\rho(D_1)$, where $\rho$ is an orientation reversing involution of $S^2$ whose fixed point set is $\p D_1=\p D_2$.
For a smooth map $\beta\colon (D^2,\p D^2)\to (M,L)$, its \emph{double} is defined by
$$
\beta^\sharp\colon S^2\to M,\quad \beta^\sharp(z)=\begin{cases}
	\beta(z), &z\in D_1\cong D^2, \\
	R(\beta(\rho(z))), &z\in D_2.
\end{cases} 
$$
Then one can check that $c_1(\beta^\sharp)=\mu_L(\beta)$ and $\ow(\beta^\sharp)=2\cdot \ow(\beta)$. Since $M$ is monotone, $L$ is monotone as well.
\end{proof}
Throughout this paper, the symplectic manifold $S^2\times S^2$ is equipped with the split symplectic form $\ow\oplus \ow$, where $\ow$ is a Euclidean area form on $S^2$.
This symplectic form is monotone, and hence every real Lagrangian in $S^2\times S^2$ is monotone.
%Then $(S^2\times S^2,\ow \oplus \ow)$ is monotone, i.e., $c_1(S^2\times S^2)=2[\ow\oplus \ow]$.
% and hence any real Lagrangians in $S^2\times S^2$ are \emph{monotone} meaning that there exists $C>0$ such that $\mu_L(\beta)=C\cdot \ow(\beta)$ for all $\beta\in \pi_2(S^2\times S^2,L)$.
\subsection{Topology of antisymplectic involutions of $S^2\times S^2$}
We fix generators of $H_2(S^2\times S^2)\cong \Z^2$,
$$
A_1=[S^2\times \{\pt\}]\quad\text{and}\quad A_2=[\{\pt\}\times S^2].
$$
The following simple topological result plays a crucial role. This result says that any  antisymplectic involution $R$ of $S^2\times S^2$ with fixed point set $\Fix(R)\cong T^2$ is \emph{homologically} the standard antisymplectic involution $R_{\Cl}$ for the Clifford torus $\mathbb{T}_{\Cl}$.
\begin{lemma}\label{lem: antisymp_matrix}
	Let $R$ be an antisymplectic involution on $S^2\times S^2$ whose fixed point set $\Fix(R)$ is diffeomorphic to $T^2$. Then the map $R_*$ induced in homology $H_2(S^2\times S^2)$ is given by $R_*A_i=-A_i$ for $i=1,2$.
\end{lemma}	
\begin{proof}
Write $R_*=\begin{pmatrix}
	a_1 & a_2\\
	b_1 & b_2
\end{pmatrix}$ for the induced map of $R$ on $H_2(S^2\times S^2)\cong \Z^2$. It  satisfies $R_*^2=\id$ and 
\begin{equation*}
R_*(A_1+A_2)=-(A_1+A_2) \iff \begin{pmatrix}
	a_1+a_2 \\
	b_1+b_2
\end{pmatrix}=\begin{pmatrix}
	-1 \\
	-1
\end{pmatrix}.
\end{equation*}
The second condition holds since $R$ is antisymplectic. Since $R$ is orientation-preserving, $R_*$ preserves the intersection form of $S^2\times S^2$. In particular, for $i=1,2$ we obtain 
\begin{align*}
	0 &= A_i\bullet A_i = R_*A_i\bullet R_*A_i = \begin{pmatrix}
		a_i \\
		b_i
	\end{pmatrix} \bullet \begin{pmatrix}
		a_i \\
		b_i
	\end{pmatrix}=2a_ib_i.
\end{align*}
Hence, $R_*$ must be either $\begin{pmatrix}
	-1 & 0 \\
	0 & -1 
\end{pmatrix}$ or $\begin{pmatrix}
	0 & -1 \\
	-1 & 0
\end{pmatrix}$, and the same result holds for $H_2(S^2\times S^2;\Q)$. The Lefschetz fixed point theorem \cite[Section 1.3]{realenriq} implies that
\begin{align*}
	\chi(\Fix(R)) &= \sum_{i=0}^4(-1)^i\tr\Big[R_*\colon H_i(S^2\times S^2;\Q)\to H_i(S^2\times S^2;\Q)\Big] \\	
&=	2+\tr\Big[R_*\colon H_2(S^2\times S^2;\Q)\to H_2(S^2\times S^2;\Q)\Big].
\end{align*}
Since $\chi(T^2)=0$, the lemma follows.
\end{proof}

\begin{remark}
By the above proof, for given antisymplectic involution $R$ of $S^2\times S^2$ with possibly empty fixed point set, the induced map $R_*$ on $H_2(S^2\times S^2)$ is either
$$
I_1=\begin{pmatrix}
	-1 & 0 \\
	0 & -1 
\end{pmatrix}\quad \text{or}\quad I_2=\begin{pmatrix}
	0 & -1 \\
	-1 & 0
\end{pmatrix}.
$$
It is known that any real Lagrangian in $S^2\times S^2$ is diffeomorphic to either $T^2$ or $S^2$, see \cite[Proposition B]{KimUniq}. Hence, if $\Fix(R)$ is nonempty we obtain
		\begin{itemize}
		\item $\Fix(R)$ is diffeomorphic to $T^2$ if and only if $R_*=I_1$.
		\item $\Fix(R)$ is diffeomorphic to $S^2$ if and only if $R_*=I_2$.
	\end{itemize}
%	In particular, Theorem \ref{thm: thmA} says that the induced map of an antisymplectic involution classifies real Lagrangian submanifolds in $S^2\times S^2$ up to Hamiltonian isotopy. %Hence, a natural question arises: are there two antisymplectic involutions such that the fixed point sets are Hamiltonian isotopic but the induced maps are different?
\end{remark}
\subsection{Gromov's foliation theorem}\label{sec: Gromov}
We recall the celebrated Gromov foliation theorem in $S^2\times S^2$, see \cite[Theorem~2.4.$A_1$]{Gromov}.
Let $\mathcal{J}$ denote the \emph{space of compatible almost complex structures} on $S^2\times S^2$.
We emphasize that for the following result it is crucial that the symplectic form $\omega \oplus \omega$ is monotone.
\begin{theorem}[Gromov]\label{thm: gromov}
For every $J\in \mathcal{J}$ there exist two transversal foliations $\mathcal{F}_1$ and $\mathcal{F}_2$ of $S^2\times S^2$ whose leaves are (unparametrized) embedded $J$-holomorphic spheres in the homology class $A_1$ and $A_2$, respectively.
\end{theorem}
Here, by \emph{transversal} foliations $\mathcal{F}_1$ and $\mathcal{F}_2$, we mean that any two leaves of $\mathcal{F}_1$ and $\mathcal{F}_2$ are transverse. By positivity of intersections \cite[Section 2.6]{McSalJholo}, the two leaves intersects transversely at a single point.
\begin{remark}\label{rem: adj}\
\begin{enumerate}
	\item[(i)] Actually, Theorem \ref{thm: gromov} holds for \emph{tame} almost complex structures.
	\item[(ii)]  By topological argument as in \cite[Remark 2.6 (a)]{CS}, any smooth $F$-fibration of $S^2\times S^2$ over a closed surface $B$ must satisfy $B\cong F\cong S^2$. Gromov's foliations yield plenty of symplectic $S^2$-fibrations of $S^2\times S^2$ over $S^2$ \cite[Proposition 4.1]{McduffRational}. See also Example \ref{ex: cliff} for an easy illustration.
	\item[(iii)] The embeddeness of the leaves of a Gromov's foliation follows from the adjunction inequality, see \cite[Theorem 2.6.4]{McSalJholo}. More precisely, any $J$-holomorphic sphere of $S^2\times S^2$ in the homology class $A_i$ is embedded.
\end{enumerate}
\end{remark}
\subsection{The neck-stretching for a real Lagrangian torus}
Let $R$ be an antisymplectic involution of $S^2\times S^2$. A compatible almost complex structure $J$ on $S^2\times S^2$ is called \emph{$R$-anti-invariant} if $J=-R^*J$.
We abbreviate by $\mathcal{J}_R$ the \emph{space of $R$-anti-invariant compatible almost complex structures} on $S^2\times S^2$, which is nonempty and contractible, see~\cite[Proposition 1.1]{Welsch}.

The following is an application of a neck-stretching argument combined with Gromov's theorem, which is a version of Dimitroglou Rizell--Goodman--Ivrii \cite[Theorem~C]{DGI}.
See also a related result of Welschinger \cite[Theorem 1.3]{Wel_Eff}.
 \begin{proposition}\label{prop: displace}
Let $L=\Fix(R)$ be a real Lagrangian torus in $S^2\times S^2$ for an antisymplectic involution $R$. Then there exists $J\in \mathcal{J}_R$ and a $J$-holomorphic sphere in $S^2\times S^2$ which represents the homology class $A_1$ and is disjoint from $L$.
\end{proposition}
%Before giving a proof, we explain a geometric setup for the neck-stretching.
The rest of this section is devoted to the proof of this statement.
Below, we mainly follow the descriptions given in \cite[Sections~2, 3, and 4]{DGI}. \vspace{-0.3cm}

 \paragraph{\it The real splitting construction.}
We explain the splitting construction for a real Lagrangian torus $L=\Fix(R)$ in $S^2\times S^2$, which matches well with the antisymplectic involution $R$.
We refer to \cite[Section~2]{DGI}, \cite[Example~2.5]{CM} or \cite[Example~1.3.1]{EGH} for the split construction for a Lagrangian.
 
 Fix a flat metric on the torus $L$ and write $(\boldsymbol{\theta},{\bf p})=(\theta_1,\theta_2,p_1,p_2)\in T^*L\cong T^2\times \R^2$ for the coordinates.
 Let $\lambda_{\can}=p_1d\theta_1+p_2d\theta_2$ be the Liouville form on $T^*L$.
 The cotangent bundle $(T^*L,d\lambda_{\can})$ carries the canonical antisymplectic involution
 $R_{\can}(\boldsymbol{\theta},{\bf p})=(\boldsymbol{\theta},-{\bf p})$, 
 which is \emph{exact}, i.e., $R_{\can}^*\lambda_{\can}=-\lambda_{\can}$, and is an isometry. 
 This map restricts to the (strict) anticontact involution $R_{\can}|_{S^*L}$ on the unit cotangent bundle $S^*L$ equipped with the contact form ${\alpha}=\lambda_{\can}|_{S^*L}$. This involution extends to the exact antisymplectic involution on the symplectization $(\R\times S^*L,d(e^t\alpha))$~by
 $$
 R_{\can}^{\cyl}(t,\boldsymbol{\theta},{\bf p}) = (t,\boldsymbol{\theta},-{\bf p}).
 $$
%By abuse of notation, we denote this involution on $\R\times S^*L$ by $R_{\can}$.
For $r>0$ we let $T^*_r L=\{(\boldsymbol{\theta},{\bf p})\in T^*L \mid \|{\bf p}\|\le r\}$.
By the equivariant Weinstein neighborhood theorem \cite[Theorem 2]{Meyer}, there exists a symplectic embedding
\begin{equation}\label{eq: symp_emb}
\Psi\colon (T^*_{4\epsilon}L,d\lambda_{\can}) \longhookrightarrow (S^2\times S^2,\ow\oplus \ow)	
\end{equation}
such that $\Psi({\bf 0}_L)=L$ and $R\circ \Psi=\Psi \circ R_{\can}$. Together with the exact symplectomorphism
\begin{equation}\label{eq: identification}
(\R\times S^*L, d(e^t\alpha)) \overset{\cong}{\longrightarrow} (T^*L\setminus {\bf 0}_{L},d\lambda_{\can}),\qquad (t,\boldsymbol{\theta},{\bf p}) \longmapsto (\boldsymbol{\theta}, e^t{\bf p}),
\end{equation}
which identifies the antisymplectic involutions $R_{\can}^{\cyl}$ and $R_{\can}$, we see that $S^2\times S^2\setminus L$ and $T^*L$ have a cylindrical end over the \emph{real} contact manifold $(S^*L,\alpha,R_{\can}|_{S^*L})$.
Therefore, we obtain the \emph{real split symplectic manifold} consisting of 
%\begin{align}\label{eq: split_symp_mfd}
%	(S^2\times S^2\setminus L,\ow\oplus\ow,R)\ \sqcup\ (T^*L,d\lambda_{\can},R_{\can})
%\end{align}
%\begin{equation}
%\begin{cases}
%	(S^2\times S^2\setminus L,\ow\oplus\ow,R) & \text{(Top level)} \\
%	(\R\times S^*L,d(e^t\alpha),R_{\can}^{\cyl}) & \text{(Middle level)}	\\
%	(T^*L,d\lambda_{\can},R_{\can}) & \text{(Bottom level)}
%\end{cases}
%\end{equation}
\begin{equation}\label{eq: split_symp_mfd}
	\left\{\begin{alignedat}{11}
		&(S^2\times S^2\setminus L,\; && \ow\oplus \ow,\; && R,\quad\; && J_{\infty}&&) \qquad\qquad&& \text{(Top level)} \\
		&(\R\times S^*L, && d(e^t\alpha), && R_{\can}^{\cyl}, && J_{\cyl}&&) && \text{(Middle level)}\\
		&(T^*L, && d\lambda_{\can}, && R_{\can}, && J_{\std}&&) && \text{(Bottom level)}
	\end{alignedat}\right.
\end{equation}
endowed with the $R$-anti-invariant almost complex structures $J_{\infty}$, $J_{\cyl}$, and $J_{\std}$, which are explained below.\vspace{-0.3cm}
\paragraph{\it A neck-stretching family of $R$-anti-invariant almost complex structures.} A compatible almost complex structure $J$ on the symplectization $(\R\times S^*L,d(e^t\alpha))$ is called \emph{cylindrical} if $J$ is $\R$-translation invariant, $J(\p_t)=\mathcal{R}_{\alpha}$, and $J(\ker \alpha)=\ker\alpha$.
Here, $\mathcal{R}_\alpha$ is the \emph{Reeb vector field} on $(S^*L,\alpha)$ uniquely determined by the equations $d\alpha(\mathcal{R}_{\alpha},\cdot)=0$ and $\alpha(\mathcal{R}_{\alpha})=1$.
We recall the construction in \cite[Section~4]{DGI} of the specific almost complex structures on the split symplectic manifold \eqref{eq: split_symp_mfd}.
We refer to  \cite[Lemma~4.1]{DGI} for details.

Using the identification \eqref{eq: identification}, we define the cylindrical compatible almost complex structure $J_{\cyl}$ on $(\R\times S^*L, d(e^t\alpha))\cong (T^*L\setminus {\bf 0}_L,d\lambda_{\can})$ by
\begin{align*}
J_{\cyl}\p_{\theta_i} &= - \|{\bf p}\|\p_{p_i},\, \qquad \text{for $(\boldsymbol{\theta},{\bf p})\in T^*L\setminus {\bf 0}_L \cong \R\times S^*L$\; and\; $i=1,2$}. \\
\intertext{We consider the \emph{tame} almost complex structure $J_{\std}$ on $(T^*L,d\lambda_{\can})$ given by}
J_{\std}\p_{\theta_i} &=-f(\|{\bf p}\|)\p_{p_i},\quad \text{for $(\boldsymbol{\theta}, {\bf p})\in T^*L$\; and\; $i=1,2$},	
\end{align*}
where $f\colon [0,\infty) \to [\epsilon,\infty)$ is a smooth function such that
\begin{itemize}
	\item $f'(t)\ge 0$ for $t\ge 0$,
	\item $f(t)=\epsilon$ for $0\le t\le \epsilon$, and
	\item $f(t)=t$ for $t\ge 2\epsilon$.
\end{itemize}
Here, $\epsilon>0$ is chosen such that \eqref{eq: symp_emb} exists.
On $T^*L\setminus T^*_{2\epsilon}L$, the almost complex structure $J_{\std}$ agrees with the cylindrical almost complex structure $J_{\cyl}$.
We can readily check that $J_{\std}$ and $J_{\cyl}$ are \emph{$R$-anti-invariant}, i.e.,
$$
J_{\std}=-R_{\can}^*J_{\std}, \qquad J_{\cyl} = - (R_{\can}^{\cyl})^*J_{\cyl}.
$$
Take the neighborhood $U:=\Psi(T^*_{4\epsilon}L)$ of $L$ in $S^2\times S^2$, where $\Psi\colon T^*_{4\epsilon}L\hookrightarrow S^2\times S^2$ is the symplectic embedding from \eqref{eq: symp_emb}.
Let $\mathcal{J}_R(S^2\times S^2\setminus L)$ be the \emph{space of $R$-anti-invariant compatible almost complex structures} on $S^2\times S^2\setminus L$. The set
$$
\mathcal{J}^{\cyl}_R(S^2\times S^2\setminus L) = \{J_\infty\in \mathcal{J}_R(S^2\times S^2\setminus L) \mid \text{$\Psi^*J_\infty=J_{\cyl}$ on $U\setminus L$}\},
$$
is nonempty and contractible.

Suppose that we are given $J_\infty\in \mathcal{J}_R^{\cyl}(S^2\times S^2\setminus L)$.
%For each $\tau\ge 0$ we define the diffeomorphism
%\begin{align*}
%\Phi_\tau\colon [\log 2\epsilon ,\log 4\epsilon)\times S^*L &\longrightarrow [\log 2\epsilon ,\log 4\epsilon +\tau)\times S^*L\\
%(t,\boldsymbol{\theta},{\bf p}) &\longmapsto (\phi_\tau(t),\boldsymbol{\theta},{\bf p}),
%\end{align*}
%where $\phi_\tau\colon [\log 2\epsilon,\log 4\epsilon) \to [\log 2\epsilon, \log 4\epsilon + \tau)$ is an orientation-preserving diffeomorphism. Note that $\Phi_{\tau}\circ R_{\can}^{\cyl}=R_{\can}^{\cyl}\circ \Phi_{\tau}$.
%One checks that $\Phi_{\tau}\circ R_{\can}^{\cyl}=R_{\can}^{\cyl}\circ \Phi_{\tau}$.
%A \emph{neck-stretching family} $\{J_\tau\}_{\tau\ge 0}$ in \cite[Section 2.5]{DGI} of $R$-anti-invariant tame almost complex structures $J_\tau$ on $S^2\times S^2$ is defined by
%$$
%J_\tau=\begin{cases}
%	J_\infty & \text{on $S^2\times S^2\setminus \Psi(T^*_{4\epsilon}L)$}, \\
%	\Phi_\tau^*J_{\cyl} & \text{on $\Psi(T_{4\epsilon}^*L\setminus T_{2\epsilon}^*L)$},\\
%	J_{\std} & \text{on $\Psi(T^*_{2\epsilon}L)$},
%\end{cases}
%$$
%where the identifications \eqref{eq: symp_emb} and \eqref{eq: identification} are used in $J_{\std}$ and $\Phi_\tau^*J_{\cyl}$.
A \emph{neck-stretching family} $\{J_\tau\}_{\tau\ge 0}$ in \cite[Section 2.5]{DGI} of $R$-anti-invariant tame almost complex structures $J_\tau$ on $S^2\times S^2$ is defined by
$$
J_\tau=\begin{cases}
	J_\infty & \text{on $S^2\times S^2\setminus \Psi(T^*_{4\epsilon}L)$}, \\
	\Phi_\tau^*J_{\cyl} & \text{on $\Psi(T_{4\epsilon}^*L\setminus T_{2\epsilon}^*L)$},\\
	J_{\std} & \text{on $\Psi(T^*_{2\epsilon}L)$},
\end{cases}
$$
where 
\begin{itemize}
	\item The map $\Phi_\tau$ is a diffeomorphism
\begin{align*}
\Phi_\tau\colon [\log 2\epsilon ,\log 4\epsilon)\times S^*L &\longrightarrow [\log 2\epsilon ,\log 4\epsilon +\tau)\times S^*L\\
(t,\boldsymbol{\theta},{\bf p}) &\longmapsto (\phi_\tau(t),\boldsymbol{\theta},{\bf p}),
\end{align*}
induced by an orientation-preserving diffeomorphism
$$
\phi_\tau\colon [\log 2\epsilon,\log 4\epsilon) \longrightarrow [\log 2\epsilon, \log 4\epsilon + \tau).
$$
	\item The identifications \eqref{eq: symp_emb} and \eqref{eq: identification} are used in $J_{\std}$ and $\Phi_\tau^*J_{\cyl}$.
\end{itemize}
%For each $\tau\ge 0$ the almost complex structure $J_\tau$ on $S^2\times S^2$ is \emph{$R$-anti-invariant}.
\begin{remark}
In the limit $\tau\to \infty$ of the quadruple $(S^2\times S^2,\ow\oplus\ow,R,J_\tau)$ we obtain the real split symplectic manifold \eqref{eq: split_symp_mfd}, see for instance \cite[Section~2.7]{CMcompact} or \cite[Section~3.4]{BEHWZ}.
\end{remark}
We are now in position to prove Proposition~\ref{prop: displace}. The SFT analysis that we will use is carried out in \cite[Sections~2 and 3]{DGI} and we closely follow their arguments adapted to our purpose.
\begin{proof}[Proof of Proposition~\ref{prop: displace}]
%Following \cite[Section 2]{DGI} or \cite[Example 1.3.1]{EGH}, we consider the split symplectic manifold whose top and bottom levels are identified with $(S^2\times S^2\setminus L,\ow\oplus \ow)$ and $(T^*L,d\lambda_{\can})$ having a cylindrical end over the unit cotangent bundle $(S^*L,\alpha=\lambda_{\can}|_{S^*L})$, respectively.
%Here, $\lambda_{\can}$ denotes the canonical Liouville form on $T^*L$ and $L\cong T^2$ is equipped with the flat metric.
Pick a regular $J_\infty\in \mathcal{J}_R^{\cyl,\, \reg}(S^2\times S^2\setminus L)$ as in Lemma~\ref{lem: equitrans} and consider the associated neck-stretching family $\{J_\tau\}_{\tau\ge 0}$ of $R$-anti-invariant tame almost complex structures $J_\tau$ on $S^2\times S^2$.
It follows from Gromov's result together with positivity of intersections that for any $\tau\ge 0$ there exists a unique embedded $J_\tau$-holomorphic sphere in the homology class $A_1$ passing through any given point in $S^2\times S^2$.
We can therefore apply the SFT compactness theorem \cite[Theorem 2.2]{DGI} or \cite{BEHWZ,CMcompact} to these spheres to obtain a limit holomorphic building in the real split symplectic manifold $(S^2\times S^2\setminus L, \ow\oplus \ow,R)\sqcup (T^*L,d\lambda_{\can},R_{\can})$ given in \eqref{eq: split_symp_mfd}.
Note that the limit holomorphic building possibly consists of components in the middle levels $(\R\times S^*L)\sqcup \cdots \sqcup (\R\times S^*L)$.
By convexity and exactness of $(T^*L,d\lambda_{\can})$ and $(\R\times S^*L,d(e^t\alpha))$, the limit holomorphic building must contain at least one component in the top level $S^2\times S^2\setminus L$.
\medskip\\
%We perform a neck-stretching along the contact hypersurface in $S^2\times S^2$ that is identified with the unit cotangent bundle $S^*L$ of $L$ and analyze a limit building. Note that the contact hypersurface $S^*L$ is of Morse--Bott. 
{\bf Claim.} \emph{If a limit holomorphic building is broken, then its top level consists of two simple $J_\infty$-holomorphic planes in $S^2\times S^2\setminus L$ of index 1.}

\medskip In \cite[Proposition~3.5]{DGI}, a similar version of the claim is obtained under the choice of a regular $J_\infty\in \mathcal{J}^{\cyl}(S^2\times S^2\setminus L)$ for \emph{all} simple punctured $J_\infty$-holomorphic spheres in $S^2\times S^2\setminus L$, where $L$ is not necessarily monotone.
The generic choice of $J_\infty$ was crucially required to guarantee the index non-negativity result for punctured spheres in $S^2\times S^2\setminus L$.
In our case, the monotonicity of $L$ controls the index of punctured spheres.\medskip\\
\textit{\textbf{Step 1.}} \emph{Any punctured $J_\infty$-holomorphic sphere in $S^2\times S^2\setminus L$ satisfies $\ind(u)\ge 1$.}

Fixing a symplectic trivialization $\Phi$ of the contact structure $\ker \alpha$ on $S^*L$, we denote by $\CZ^\Phi(\Gamma)$ the Conley--Zehnder index of a Morse--Bott manifold $\Gamma$ of periodic Reeb orbits in $S^*L$ and by $c_{1,\rel}^\Phi(\cdot )$ the relative first Chern number, see \cite[Section~3.1]{DGI}.
By \cite[Equations~(1) and (2) in Section~3.1]{DGI}, the index of a punctured $J_\infty$-holomorphic sphere $u$ in $S^2\times S^2\setminus L$ having $\ell\ge 1$ negative punctures asymptotic to periodic orbits in Morse--Bott manifolds $\Gamma_1,\dots,\Gamma_\ell$ is given by
\begin{align*}
	\ind(u) &= -2+\ell -\sum_{i=1}^{\ell}(\CZ^{\Phi}(\Gamma_i^-)-1) + 2c_{1,\rel}^{\Phi}(u) \\
	&= -2 + \ell +\mu_L(\overline{u}),
\end{align*}
where $\overline{u}$ is a surface in $S^2\times S^2$ with boundary on $L$ which is the boundary compactification of $u$ obtained by adding (geodesic) circles in $L$ corresponding to the asymptotic orbits.
Since $L$ is orientable, $\overline{u}$ has Maslov number $\mu_L(\overline{u})\in 2\Z$.
By the monotonicity of $L$ together with the fact that $u$ is non-constant, we have $\mu_L(\overline{u})\ge 2$. Hence, we deduce that
$$
\ind(u)=-2+\ell+\mu_L(\overline{u})\ge 1,
$$
which completes Step 1.\medskip\\
\textit{\textbf{Step 2.}} \emph{Proof of the claim.}

This essentially follows from the proof of \cite[Proposition 3.5]{DGI}. Indeed, that proof yields the identity
$$
\sum_{i=1}^{N_1}\ind(v_i)-\sum_{i=1}^{N_2}\chi(w_i)=2,
$$
where $v_1,\dots,v_{N_1}$ and $w_1,\dots, w_{N_2}$ denote the components of the limit holomorphic building in the top level $S^2\times S^2\setminus L$ and the remaining levels, respectively.
Here, $\chi(w_i)$ is the Euler characteristic of the domain of $w_i$.
Since there are no contractible periodic orbits in $S^*L$, a pseudoholomorphic \emph{plane} in $T^*L$ or $\R\times S^*L$ asymptotic to a periodic orbit cannot exist, which shows $\chi(w_i)\le 0$.
By Step 1, we know that $\ind(v_i)\ge 1$.
Since the components of the limit holomorphic building in the split symplectic manifold $(S^2\times S^2\setminus L)\sqcup T^*L$ glue together to form a sphere, we conclude that the top level of the limit holomorphic building consists of two planes of index 1. 
The simplicity of the plane follows from \cite[Lemmata~3.3 and 3.4]{DGI}. Hence, the claim follows. \medskip
%To see the embeddedness property, recall that the building is a limit of \emph{embedded} $J$-holomorphic spheres. By the $C^\infty_{\loc}$-convergence nature of SFT compactness together with the fact that the plane is simple (i.e., it is not multiply covered), we conclude that the plane is embedded as well. See ... \\\\

We now crucially use that our $J_\infty \in \mathcal{J}_R^{\cyl,\, \reg}(S^2\times S^2\setminus L)$ is regular.
By the claim, if the limit building is broken, then its top level consists of two simple $J_\infty$-holomorphic planes in $S^2\times S^2\setminus L$ of index 1.
Such planes are contained in the moduli space of index 1 simple $J_\infty$-holomorphic planes in $S^2\times S^2\setminus L$, which is a smooth manifold of dimension 1.
Since the total collection of components of \emph{broken} limit buildings in the top level $S^2\times S^2\setminus L$ forms a subset of dimension at most 3 in $S^2\times S^2\setminus L$, we conclude that there must be a \emph{non-broken} limit building, that is an embedded $J_\infty$-holomorphic sphere $u$ in $S^2\times S^2\setminus L$ whose homology class represents $A_1$.
It remains to show that $u$ can be seen as a $J$-holomorphic sphere in $S^2\times S^2$ for some $J\in \mathcal{J}_R$.
Since the image of $u$ is a compact set in $S^2\times S^2\setminus L$, we can choose an $R$-invariant neighborhood of $L$ in $S^2\times S^2$ which is disjoint from $u$. 
We then modify $J_\infty$ on this neighborhood so that the resulting one can be extended to an $R$-anti-invariant compatible almost complex structure $J$ defined on $S^2\times S^2$, namely $J\in \mathcal{J}_R$. This completes the proof.
%More precisely, we can find $x\in S^2\times S^2\setminus L$ which is disjoint from the total collection by dimensional reasons. By Gromov's foliation there exists an embedded $J_\infty$-holomorphic sphere in $S^2\times S^2$ which represents the homology class $A_1$ and passes through $x$.
%conclude that there must be a \emph{non-broken} limit building, that is an embedded $J_\infty$-holomorphic sphere in $S^2\times S^2\setminus L$ whose homology class represents $A_1$.}
\end{proof}
Following ideas of \cite[Section 3]{KimChek}, we show the equivariant transversality for simple pseudoholomorphic planes in $S^2\times S^2\setminus L$ asymptotic to a periodic orbit. 
\begin{lemma}\label{lem: equitrans}
	There exists a Baire subset $\mathcal{J}_R^{\cyl,\, \reg}(S^2\times S^2\setminus L)\subset \mathcal{J}_R^{\cyl}(S^2\times S^2\setminus L)$ which has the property that every $J_\infty\in \mathcal{J}_R^{\cyl,\, \reg}(S^2\times S^2\setminus L)$ is regular for simple $J_\infty$-holomorphic planes in $S^2\times S^2\setminus L$ asymptotic to a periodic orbit in $S^*L$.
\end{lemma}
\begin{proof}
We first observe that every simple $J_\infty$-holomorphic plane $u\colon \C\cong S^2\setminus\{\infty\}\to S^2\times S^2\setminus L$ is not $R$-invariant for topological reasons, namely $\im u\ne R(\im u)$. Assume to the contrary that $\im u=R(\im u)$. Let $\rho_0$ be an antiholomorphic involution on $S^2$ leaving $\infty\in S^2$ invariant. Since $u$ and $R\circ u\circ \rho_0|_{S^2\setminus \{\infty\}}$ are simple $J_\infty$-holomorphic planes whose images coincide, there exists $\sigma\in \Aut(S^2)$ such that $\sigma(\infty)=\infty$ and
\begin{equation}\label{eq: identity}
u= R\circ u \circ \rho_0 \circ \sigma|_{S^2\setminus \{\infty\}},	
\end{equation}
see \cite[Theorem~3.7]{Nelson} and \cite[Corollary~2.5.4]{McSalJholo}.
By applying \eqref{eq: identity} twice, we obtain that $u\circ (\rho_0\circ \sigma|_{S^2\setminus \{\infty\}})^2=u$.
Since $u$ is simple, $\rho_0\circ \sigma|_{S^2\setminus \{\infty\}}$ is an involution of $\C$ and hence has a fixed point, say $z_0\in \C$.
We then see that $u(z_0)\in \Fix(R)=L$, which is a contradiction.
Since $J_\infty$-holomorphic planes (asymptotic to periodic orbits) are not $R$-invariant and are proper, the lemma follows almost verbatim from the proof of \cite[Theorem~3.9]{KimChek}. 
We emphasize that our $J_\infty\in \mathcal{J}^{\cyl,\, \reg}_R(S^2\times S^2\setminus L)$ is required to satisfy $\Psi^*J_\infty=J_{\cyl}$ on $U\setminus L$, but this is not problematic.
By convexity, $u$ cannot be entirely contained in a cylindrical end $U\setminus L$ since otherwise it violates the maximum principle.
Therefore, the transversality argument works by perturbing $J_\infty$ outside of $U\setminus L$.
\end{proof}
\begin{remark}
Since $\mathcal{J}_R$ is closed in $\mathcal{J}$ (in the $C^\infty$-topology), the Baire category theorem and Lemma \ref{lem: equitrans} imply that $\mathcal{J}_R^{\reg}$ is dense in $\mathcal{J}_R$, and hence not empty.
\end{remark}
%\begin{remark}
%{\color{blue}	The neck-stretching argument in this section can be generally applied to ...}
%\end{remark}

\subsection{Cieliebak--Schwingenheuer's criterion}\label{sec: CS}
We give a review on the Cieliebak--Schwingenheuer criterion \cite[Theorem 1.1]{CS} when a monotone Lagrangian torus in $S^2\times S^2$ is Hamiltonian isotopic to the Clifford torus $\mathbb{T}_{\Cl}$. A monotone Lagrangian torus $L$ in $S^2\times S^2$ is called \emph{fibered} if there exist a foliation $\mathcal{F}$ of $S^2\times S^2$ by symplectic 2-spheres in the homology class $A_2$ and a symplectic 2-sphere $\Sigma$ in the homology class $A_1$ such~that
\begin{itemize}
	\item $\Sigma$ is transverse to the leaves of $\mathcal{F}$.
	\item $\Sigma$ is disjoint from $L$.
	\item The leaves of $\mathcal{F}$ intersect $L$ in a circle or not at all.
\end{itemize}
In this case, we say that $L$ is \emph{fibered by $\mathcal{F}$ and $\Sigma$}.
It is a highly non-trivial result that any monotone Lagrangian torus in $S^2\times S^2$ is fibered. This result, which originally goes back to Ivrii's thesis \cite{Ivrii}, is proved by Dimitroglou Rizell--Goodman--Ivrii \cite[Theorem D]{DGI} based on the neck-stretching argument \cite{BEHWZ}.

Consider a monotone Lagrangian torus $L$ in $S^2\times S^2$ which is fibered by $\mathcal{F}$ and $\Sigma$. Each leaf of the foliation $\mathcal{F}$ intersecting $L$ is written as a union of two closed discs glued along the embedded loop given by the intersection of $L$ and the leaf. Hence, the discs intersecting $\Sigma$ form a solid torus $T$ with boundary $\p T=L$. Note that the discs which do not intersect $L$ also define a solid torus $T'$ with $\p T'=L$.

\begin{example}[Clifford torus]\label{ex: cliff}
Let $i\oplus i$ be the split complex structure on $S^2\times S^2$.
Gromov's foliations are given by $\mathcal{F}_1$ and $\mathcal{F}_2$ whose leaves are the holomorphic spheres $F_{1,y}:=S^2\times \{y\}$ for $y\in S^2$ and $F_{2,x}:=\{x\}\times S^2$ for $x\in S^2$, respectively.
The associated symplectic $S^2$-fibration is defined as follows.
Fix one leaf $F_{1,y_0}$ of $\mathcal{F}_1$. Since each leaf $F_{2,x}$ of $\mathcal{F}_2$ intersects $F_{1,y_0}$ transversely at a unique point $(x,y_0)$, we can define a symplectic $S^2$-fibration of $S^2\times S^2$ by sending $F_{2,x}$ to $(x,y_0)\in F_{1,y_0}\cong S^2$.
In this case, this is the projection $\pi(x,y)=x$ of $S^2\times S^2$ onto the first $S^2$-factor.
One checks that the Clifford torus $\mathbb{T}_{\Cl}=S^1\times S^1$ is fibered by $\mathcal{F}_2$ and $F_{1,y_0}$.
Each fiber $F_{2,x}\cong S^2$ intersecting $\mathbb{T}_{\Cl}$ is a union of two holomorphic discs of Maslov index 2 with boundary on $\mathbb{T}_{\Cl}$.
The $S^1$-family of the discs intersecting $F_{1,y_0}$ forms a solid torus $T$ with $\p T=\mathbb{T}_{\Cl}$. It is worth noting that we can find another symplectic section $F_{1,y_1}$ for some $y_1\in S^2$ (for example, the antipodal point of $y_0\in S^2$) which does not intersect the solid torus $T$.
	
\end{example}

The criterion of Cieliebak--Schwingenheuer says that the existence of the second \emph{nice} symplectic 2-sphere in the homology class $A_1$ guarantees that a given monotone Lagrangian torus in $S^2\times S^2$ is Hamiltonian isotopic to the Clifford torus. 
\begin{theorem}[Cieliebak--Schwingenheuer]\label{thm: CS}
Let $L$ be a monotone Lagrangian torus in $S^2\times S^2$ which is fibered by $\mathcal{F}$ and $\Sigma$. Suppose that there exists a second symplectic 2-sphere $\Sigma'$ of $S^2\times S^2$ in the homology class $A_1$ such that
\begin{itemize}
	\item $\Sigma'$ is transverse to the leaves of $\mathcal{F}$,
	\item $\Sigma'$ is disjoint from $\Sigma$ and the solid torus $T$.
\end{itemize}
Then $L$ is Hamiltonian isotopic to the Clifford torus~$\mathbb{T}_{\Cl}$.
\end{theorem}

The proof is based on a sophisticated version of the Lalonde--Mcduff inflation procedure \cite{LM}.
\begin{remark}
The converse of Theorem \ref{thm: CS} obviously holds, namely if a monotone Lagrangian torus $L$ in $S^2\times S^2$ is Hamiltonian isotopic to the Clifford torus $\mathbb{T}_{\Cl}$, then $L$ is fibered by some $\mathcal{F}$, $\Sigma$, and the second symplectic 2-sphere $\Sigma'$ as well.
\end{remark}

\section{Proof of the Main Theorem}\label{sec: sec3}
Throughout this section, we let $L=\Fix(R)$ be a real Lagrangian torus in $S^2\times S^2$ for some antisymplectic involution $R$ of $S^2\times S^2$. Recall that every $J$-holomorphic sphere of $S^2\times S^2$ in the homology class $A_i$ for $i=1,2$ is embedded, see Remark \ref{rem: adj}. We start with the following simple observation.
\begin{lemma}\label{lem: circleorempty}
Let $J\in \mathcal{J}_R$ and $i=1,2$. Suppose that $u\colon S^2 \to S^2\times S^2$ is a $J$-holomorphic sphere in the homology class $A_i$. If $u$ intersects $L$, then the embedded 2-sphere $\im u$ is $R$-invariant and $\im u\cap L$ is diffeomorphic to $S^1$.
\end{lemma}
\begin{proof}
Consider the $J$-holomorphic sphere in $S^2\times S^2$ given by
$$
u':=R\circ u \circ \rho,
$$
where $\rho(z)=\bar{z}^{-1}$ denotes the antiholomorphic involution of $S^2\cong \C\cup \{\infty\}$ with $\Fix(\rho)\cong S^1$.
Here, $u'$ is $J$-holomorphic since $J$ is $R$-anti-invariant.
By Lemma \ref{lem: antisymp_matrix}, $u'$ represents the homology class $A_i$.
By positivity of intersections, $u'$ must be one leaf of Gromov's foliation $\mathcal{F}_i$ associated to $J$ as in Theorem \ref{thm: gromov}.
Pick any $x\in \im u\cap L$. Since $u(z)=x$ for some $z\in S^2$, we see that
	$$
	u'(\rho(z))=R\circ u\circ \rho(\rho(z))=R(u(z))=R(x)=x,
	$$
	which shows that $u'$ passes through $x$.
	This implies that $\im u=\im u'$, and hence $\im u$ is $R$-invariant as desired.
	Recalling that $u$ is embedded, the antisymplectic involution $R$ of $S^2\times S^2$ restricts to a smooth involution $\tau:=R|_{\im u}$ of the sphere $S^2\cong \im u$.
	Since $R_*A_i=-A_i$ by Lemma \ref{lem: antisymp_matrix}, the involution $\tau$ is orientation-reversing.
	Recall that every smooth involution of $S^2$ is conjugated to a map in $O(3)$, see \cite[Theorem~4.1]{Kerek}.
	Hence $\tau$ is conjugated to either the reflection or the antipodal map on $S^2$. %, see also Remark~\ref{rem: interchange}.
	Since $x\in \Fix(\tau)$ and, in particular, $\Fix(\tau)$ is nonempty, $\tau$ is conjugated to the reflection, and hence  $\Fix(\tau)=\im u \cap L\cong S^1$.
\end{proof}

\begin{remark}\label{rem: interchange}
We notice that any smooth involution $\sigma$ of $S^2$ with fixed point set $\Fix(\sigma)\cong S^1$ must interchange the two discs obtained by cutting $S^2$ along the embedded loop $\Fix(\sigma)$.
This again follows from the fact that $\sigma$ is conjugated to a map $\tau\in O(3)$. %, see \cite[Theorem 4.1]{Kerek}.
Since $\Fix(\sigma)\cong \Fix(\tau)\cong S^1$, we deduce that $\tau$ is a reflection. Hence, $\sigma$ must interchange the two discs as desired.
\end{remark}
An immediate consequence of Lemma \ref{lem: circleorempty} together with Proposition \ref{prop: displace} is that the real Lagrangian torus $L=\Fix(R)$ is fibered by Gromov's foliations $\mathcal{F}_1$ and $\mathcal{F}_2$ associated to some $J\in \mathcal{J}_R$. More precisely, we have the following.
\begin{corollary}\label{cor: fibered}
Let $L=\Fix(R)$ be a real Lagrangian torus in $S^2\times S^2$ for an antisymplectic involution $R$. Then there exist $J\in \mathcal{J}_R$ and a leaf $\Sigma\in \mathcal{F}_1$ such that $L$ is fibered by $\mathcal{F}_2$ and $\Sigma$, that~is,
\begin{itemize}
	\item $\Sigma$ is transverse to the leaves of $\mathcal{F}_2$.
	\item $\Sigma$ is disjoint from $L$.
	\item The leaves of $\mathcal{F}_2$ intersect $L$ in a circle or not at all.
\end{itemize}
Here, $\mathcal{F}_1$ and $\mathcal{F}_2$ denote Gromov's foliations associated to $J$.
\end{corollary}
\begin{proof}
	By Proposition \ref{prop: displace}, we can choose an embedded $J$-holomorphic sphere $\Sigma$ which represents the homology class $A_1$ and is disjoint from $L$.
	It follows from positivity of intersections that $\Sigma$ is transverse to the leaves of $\mathcal{F}_2$.
	By Lemma \ref{lem: circleorempty}, we see that the leaves of $\mathcal{F}$ intersect $L$ in a circle or not at all.
	Since $J$-holomorphic spheres are symplectic, the corollary follows.
\end{proof}
We are now in position to prove the main theorem.
\begin{proof}[Proof of the Main Theorem]
Applying Corollary \ref{cor: fibered}, choose  $J\in \mathcal{J}_R$ and a leaf $\Sigma\in \mathcal{F}_1$ such that the real Lagrangian torus $L=\Fix(R)$ is fibered by $\mathcal{F}_2$ and $\Sigma$. We parametrize $\Sigma$ by an embedded $J$-holomorphic sphere $u\colon S^2\to S^2\times S^2$ in the homology class $A_1$. Consider another embedded $J$-holomorphic sphere $u':=R\circ u \circ \rho$  which is disjoint from $L$ as well. By Lemma~\ref{lem: antisymp_matrix}, $u'$ also represents the homology class~$A_1$. We write $\Sigma'=\im u'$ for its image. By positivity of intersections, we know that
	\begin{itemize}
		\item 	$\Sigma'$ is transverse to the leaves of $\mathcal{F}_2$.
		\item We have $\Sigma=\Sigma'$ or $\Sigma\cap \Sigma'=\emptyset$.
	\end{itemize}
For each leaf $F$ of $\mathcal{F}_2$ intersecting $L$, the antisymplectic involution $R$ of $S^2\times S^2$ restricts to the orientation-reversing involution $\tau$ on the leaf $F\cong S^2$ with fixed point set $\Fix(\tau)=L\cap F\cong S^1$. The embedded loop $\Fix(\tau)$ cuts $F$ into two closed discs glued along $\Fix(\tau)$. Since $\tau$ must interchange the two discs by Remark \ref{rem: interchange}, we deduce that $\Sigma\cap F$ and $\Sigma'\cap F$ are disjoint. Hence, $\Sigma$ and $\Sigma'$ must be disjoint.
Moreover, the second symplectic sphere $\Sigma'$ is disjoint from the solid torus $T$ which is the union of one of the two discs of each leaf intersecting $\Sigma$. Now we can apply the Cieliebak--Schwingenheuer criterion (Theorem \ref{thm: CS}) to complete the proof.
\end{proof}

\subsection*{Acknowledgement}
This paper came out of ideas I learned from Kai Cieliebak. The author cordially thanks Kai Cieliebak for invaluable comments, and Felix Schlenk and Urs Frauenfelder for fruitful discussions. The author also deeply thanks the anonymous referee who pointed out a mistake in the original proof of the main theorem. This work is supported by Samsung Science and Technology Foundation under Project Number SSTF-BA1901-01. Finally, the author always thanks KIAS for providing a great place to conduct research.

\bibliographystyle{abbrv}
\bibliography{mybibfile}

\end{document}